\documentclass[12pt,oneside]{amsart} 

\usepackage[dvips]{graphics}

\usepackage{epsfig}
\usepackage{stmaryrd}

\usepackage{amsfonts}
\usepackage{amsmath, amsthm, amssymb, ulem, amscd} 

\newcommand{\stopthm}{\hfill$\square$\medskip}

\newcommand{\cF}{{\mathcal F}} 
\newcommand{\cP}{{\mathcal P}} 
\newcommand{\pa}{\partial}

\newcommand{\ep}{\epsilon}

\newcommand{\om}{\omega}

\newcommand{\Xb}{\overline{X}}

\newcommand{\cC}{\mathcal{C}}

\newcommand{\gb}{{\overline g}} 
\newcommand{\gh}{{\widehat g}}

\newcommand{\Vol}{\operatorname{Vol}}
\newcommand{\Ric}{\operatorname{Ric}}

\newcommand{\Xint}{\buildrel \circ\over X}

\theoremstyle{plain}
\newtheorem{theorem}{Theorem}[section]

\newtheorem{proposition}[theorem]{Proposition}

\theoremstyle{definition}

\theoremstyle{remark}

\numberwithin{equation}{section}

\title{A note on renormalized volume functionals}
\author{Sun-Yung Alice Chang}
\address{Department of Mathematics, Princeton University, Princeton, NJ
  08544}
\email{chang@math.princeton.edu}
\thanks{The research of the first-named author is partially
  supported by 
  NSF grant DMS-1104536.  The research of the second-named author is
  partially supported by NSF grant DMS-1008249.  The research of the 
  third-named author is partially supported by NSF grant DMS-0906035.}  
\author{Hao Fang}
\address{Department of Mathematics, University of Iowa, Maclean Hall, Iowa
  City, Iowa 52242-1419 }
\email{hao-fang@uiowa.edu}
\author{C. Robin Graham}
\address{Department of Mathematics, University of Washington,
Box 354350\\
Seattle, WA 98195-4350}
\email{robin@math.washington.edu}

\begin{document}

\maketitle

\thispagestyle{empty}

\begin{center}
{\it Dedicated to Michael Eastwood on his 60th birthday}
\end{center}

\bigskip

\section{Introduction}\label{intro}

The asymptotic expansion of the volume of an asymptotically hyperbolic
Einstein (AHE) metric defines invariants of the AHE metric and of a
metric in the induced conformal class at infinity.  These have been of
recent interest, motivated in part by the AdS/CFT correspondence in
physics.  In this paper we derive some new properties of these invariants.    

 
Let $(X^{n+1},g_+)$ be AHE with smooth conformal infinity $(M,[g])$, $M=\pa
X$.  We always assume that $X$ is connected although $\pa X$ need not be.
If $r$ is a geodesic defining function associated to a metric $g$ in 
the conformal class at infinity (see \S\ref{renormvol} for more details),
we have the following volume expansion~(\cite{G1}):  \vskip .1in  

\noindent For $n$ even,
$$ \aligned \text{Vol}_{g_{+}}(\{r > \epsilon\})  = c_0 \epsilon^{-n} &+
c_2\epsilon^{-n+2} + \cdots \\ & + 
c_{n-2}\epsilon^{-2} + L \log \frac 1\epsilon + V_{g_{+}} + o(1)\\
\endaligned $$

\noindent For $n$ odd, 
$$\aligned \text{Vol}_{g_{+}}(\{r >
\epsilon\})  = c_0 \epsilon^{-n} &+ c_2\epsilon^{-n+2} + \cdots   \\
& + c_{n-1} \epsilon^{-1} + V_{g_{+}} + o(1). \\
\endaligned $$
If $n$ is even, $L$ is independent of the choice of $g$; if $n$ is odd, 
$V_{g_{+}} $ is independent of $g$.  These are invariants of the 
conformal infinity $(M,[g])$ and the AHE manifold $(X,g_+)$, resp.  The
constants $c_{2k}$ and the renormalized volume $V_{g_+}$ for $n$ even
depend on the choice of representative metric $g$ in the conformal
infinity.  

The coefficients $c_{2k}$ and $L$ can be written as integrals over $M$ of
local 
expressions in the curvature of $g$, the so-called renormalized volume
coefficients $v^{(2k)}(g)$.  (The notation $v^{(2k)}$ is the same 
as in \cite{G1, CF}.  In \S\ref{secondvar} we also use the  
notation $v_k=(-2)^kv^{(2k)}$ of \cite{G2}.)  Changing  
perspective slightly, one realizes a given conformal  
manifold $(M^n,[g])$ as the conformal boundary of an AHE manifold $(X,g_+)$
in an asymptotic sense (see \cite{FG2}), and the $v^{(2k)}(g)$ are the
coefficients in the asymptotic expansion of the volume
form of $g_+$.  They have recently been studied in 
\cite{CF,G2}.  They are well-defined for general metrics for all $k \geq 0$
when $n$ is odd but only for $0\leq k \leq n/2$ when $n$ is even.  However,
for $n\geq 4$ even they are also defined for all $k\geq 0$ if $g$ is
locally conformally flat or conformally Einstein.    
Directly from the definition of $v^{(2k)}$ (see \S\ref{renormvol}),  we have 
$$
 c_{2k} = \frac{1}{n-2k}\int_M v^{(2k)}(g) dv_g, \qquad
0 \leq k \leq \lfloor (n-1)/2\rfloor, 
$$
$$
L = \int_M v^{(n)}(g) dv_g,\qquad n\text{ even}.
$$

When $n$ is even, $V_{g_+}=V_{g_+}(g)$ is a global quantity depending on
the choice of $g$.  Nonetheless, its change under conformal rescaling of 
$g$ 
can be expressed by an integral of a local expression over the boundary.
If $\gh=e^{2\om}g$ is a conformally related metric, then 
$$
V_{g_+}(\gh)-V_{g_+}(g)=\int_M\cP_g(\om)dv_g,
$$
where $\cP_g$ is a polynomial nonlinear differential operator whose
coefficients depend polynomially on $g$, $g^{-1}$ and the curvature of $g$
and its covariant derivatives, and whose linear part in $\omega$ and its
derivatives is $v^{(n)}(g)\om$ (see \cite{G1}).  In particular, 
$$
\pa_t V_{g_+}(e^{2t\om}g)|_{t=0} = \int_Mv^{(n)}(g)\om \,dv_g,
$$
i.e. $V_{g_+}$ is a conformal primitive of $v^{(n)}$.  

Our first result is a formula for $V_{g_+}$ for $n$ odd in terms of a 
compactification of the AHE metric $g_+$.  If $s$ is any defining function
for $\pa X$, the metric $\gb=s^2g_+$ is called a compactification of
$g_+$.  For any such compactification, $\partial X$ is an umbilic
hypersurface relative to $\gb$, i.e. its second fundamental form is a
smooth multiple of the induced metric.  We will say that $\gb$ is a totally
geodesic compactification if  
the second fundamental form of $M=\pa X$ relative to $\gb$ vanishes
identically.  If $\gb$  is any compactification, then $e^{2\om}\gb$ is
a totally geodesic compactification for many choices of $\om\in 
C^\infty(\Xb)$;  
the totally geodesic condition on $e^{2\om}\gb$ is equivalent to the
condition that the normal derivative of $\om$ at $\pa X$ be a specific 
function determined by $\gb$.  

\begin{theorem}\label{geodcomp}
If $n\geq 3$ is odd, $g_+$ is AHE, and $\gb$ is a totally geodesic 
compactification of $g_+$, then 
\begin{equation}\label{renormvolformula}
V_{g_+}=C_{n+1}\int_Xv^{(n+1)}(\gb)\,dv_{\gb},\qquad    
C_{n+1}=\frac{2^{n-1}(n+1)(\frac{n-1}{2})!\,^2}{n!}.   
\end{equation}
\end{theorem}

In the special case $n=3$, Theorem~\ref{geodcomp} follows from a result
of 
Anderson \cite{An} (see also \cite{CQY} for a different proof) expressing
the Gauss-Bonnet formula in terms of $V_{g_+}$.  Anderson showed that
\begin{equation}\label{gb1}
8\pi^2\chi(X) = \frac14\int_X |W|_{g_+}^2 dv_{g_+} +6 V_{g_+},
\end{equation}
where $W$ is the Weyl tensor and $|W|^2=W^{ijkl}W_{ijkl}$.  
On the other hand, for a totally geodesic compactification, the boundary  
term vanishes in the Gauss-Bonnet formula for the compact
manifold-with-boundary 
$(X,\gb)$.  It was observed in \cite{CGY} that in dimension $4$ the
Pfaffian is a multiple of $\frac14 |W|^2+4\sigma_2(g^{-1}P)$, where $P$  
denotes the 
Schouten tensor and $\sigma_k(g^{-1}P)$ the $k$-th elementary symmetric   
function of the eigenvalues of the endomorphism $g^{-1}P$.  Since 
$v^{(4)}(g) =\frac14 \sigma_2(g^{-1}P)$, the Gauss-Bonnet formula for 
$(X,\gb)$ can be written
\begin{equation}\label{gb2}
8\pi^2\chi(X)=\int_X \left[\tfrac14 |W|_{\gb}^2 
+16 v^{(4)}(\gb)\right]dv_{\gb}. 
\end{equation}
Comparing \eqref{gb1} and \eqref{gb2} and recalling that $\int |W|^2$ is
conformally invariant gives \eqref{renormvolformula}.     

For $n>3$ odd, a generalization of Anderson's formula expressing the
renormalized volume as a linear combination of
the Euler characteristic and the integral of a pointwise
conformal invariant has been established in~\cite{CQY}.  We do not use this
identity, but instead use the idea of  
its proof to directly relate the renormalized volume to the integral of
$v^{(n+1)}(\gb)$ for a particular totally geodesic compactification.  
The fact that \eqref{renormvolformula} then holds for any totally geodesic  
compactification follows using the result of \cite{G2} that under  
conformal change, the $v^{(2k)}$ depend on at most two derivatives of the
conformal factor.  

Our second result concerns the renormalized
volume functionals $\cF_k(g)$ defined by   
$$
\cF_k(g)=(-2)^{k}\int_Mv^{(2k)}(g)\,dv_g
$$
on the space of metrics on a connected compact manifold $M$.   
This normalization is chosen so that 
$\cF_k(g)=\int_M\sigma_k(g^{-1}P)\,dv_g$ if $g$ is locally conformally
flat; the conformal properties of the functionals
$\int_M\sigma_k(g^{-1}P)\,dv_g$ have been intensively studied during the
last decade.  In \cite{CF} it was shown
that if $2k\neq n$, then the Euler-Lagrange equation for $\cF_k$   
under conformal change subject to the constraint that the volume is  
constant is $v^{(2k)}(g)=c$.  In \cite{G2}, it was shown that if a
background metric $g_0$ in the conformal class is fixed and one writes
$g=e^{2\om}g_0$, then the Euler-Lagrange equation $v^{(2k)}(e^{2\om}g_0)=c$
is second order in $\om$ even though for $k\geq 2$, $v^{(2k)}(g)$ depends
on $2k-2$ derivatives of $g$.  

Any Einstein metric satisfies $v^{(2k)}(g)=c$, so is a critical point of
$\cF_k$.  In this paper, we  identify the second variation  
at a general critical point of $\cF_k$ under conformal change subject to 
the constant volume constraint and use this to show that Einstein metrics
of nonzero scalar curvature are 
local extrema.  Let $\cC$ denote a conformal class of metrics on $M$   
and let $\cC_1$ denote the subset of metrics of unit volume.  

\begin{theorem}\label{signs}
Let $(M,g)$ be a unit volume connected compact Einstein manifold of
dimension $n\geq 3$ with nonzero scalar  
curvature which is not isometric to $S^n$ with the standard metric
(normalized to have unit volume).  Suppose $1\leq k\leq n$ and if $n$ is
even assume that $k\neq n/2$.  
Then the second variation of $\cF_k|_{\cC_1}$,
$\left(\cF_k|_{\cC_1}\right)''$, is a definite quadratic form on 
$T_g\cC_1$ whose sign is as follows:
\begin{enumerate}
\item
Let $k<n/2$.
\begin{itemize}
\item  
If $R>0$, then $\left(\cF_k|_{\cC_1}\right)''$ is positive definite.
\item
If $R<0$, then $\left(\cF_k|_{\cC_1}\right)''$ is positive definite for $k$
odd and negative definite for $k$ even. 
\end{itemize}
\item
Let $k>n/2$.  Then all signs are reversed:
\begin{itemize}
\item  
If $R>0$, then $\left(\cF_k|_{\cC_1}\right)''$ is negative definite.
\item
If $R<0$, then  $\left(\cF_k|_{\cC_1}\right)''$ is negative definite for
$k$ odd and positive definite for $k$ even. 
\end{itemize}
\end{enumerate}
For $S^n$ with the (normalized) standard metric, the only change is that  
$\left(\cF_k|_{\cC_1}\right)''$ is semi-definite with the indicated sign
and with $n+1$-dimensional nullspace.  
\end{theorem}


Of course, one concludes from Theorem~\ref{signs} that $\cF_k|_{\cC_1}$
has a local maximum or minimum
at an Einstein metric, with sign determined as in the statement of the 
theorem.  The max-min conclusion follows also for $S^n$
since the null directions for the Hessian arise from conformal
diffeomorphisms. 

If $g$ is locally conformally flat or if $k=1$ or $2$, then
$(-2)^kv^{(2k)}(g)=\sigma_k(g^{-1}P)$.  In these cases Theorem~\ref{signs} 
follows from Theorem 2 of \cite{V}, which is concerned with the second  
variation of the functionals $\int_M \sigma_k(g^{-1}P)\, dv_g$.  
(Theorem~\ref{signs} corrects a sign error in the statement of Theorem 2 of
\cite{V} for $k>n/2$.)  Theorem~\ref{signs} indicates that for $k>2$
and $g$ not locally conformally flat, $(-2)^kv^{(2k)}(g)$ is the natural
replacement for $\sigma_k(g^{-1}P)$.
The special case $k=3$, $n>6$, of Theorem~\ref{signs} was first proved by    
Guo-Li~\cite{GL} by direct computation of $\left(\cF_3|_{\cC_1}\right)''$
from the explicit formula for $v^{(6)}(g)$.

As mentioned above, Theorem~\ref{signs} follows from a   
formula which we derive for the second variation of
$\cF_k|_{\cC_1}$ at a general critical point (Theorem~\ref{hess}).     
This second variation formula is an immediate consequence of a formula
derived in \cite{ISTY} and rederived in \cite{G2} for the first 
conformal variation of $v^{(2k)}(g)$:  by the result of \cite{CF}, 
the first conformal variation of $\cF_k$ is integration against a multiple
of $v^{(2k)}(g)$, 
so the second variation of $\cF_k$ is integration against the first
conformal variation of $v^{(2k)}(g)$.  The principal part of these
variations is a symmetric contravariant 2-tensor $L^{ij}_{(k)}(g)$ defined
by \eqref{Ldef} which was derived in \cite{ISTY} and analyzed in some
detail in \cite{G2}.  We also state a general condition in terms of 
$L^{ij}_{(k)}(g)$ which is sufficient for definiteness of the second
variation of $\cF_k|_{\cC_1}$ for non-Einstein critical points and which  
generalizes a criterion of Viaclovsky 
in the cases $k=1$, $2$ or $g$ locally conformally flat when $g$ has
(possibly nonconstant) negative scalar curvature.  

If $n$ is even, $\cF_{n/2}(g)$ is conformally 
invariant as noted above, so conformal variations of $\cF_{n/2}(g)$ 
are trivial.  A natural substitute for $\cF_{n/2}(g)$ as far 
as conformal variations is concerned is  
the renormalized volume $V_{g_+}(g)$ of an AHE metric $g_+$ with conformal
infinity $(M,[g])$.  $V_{g_+}$ is a conformal primitive  
of $v^{(n)}$ as noted above, just as $(-2)^{-k}\frac{1}{n-2k}\cF_k$ is a
conformal primitive of 
$v^{(2k)}$ if $2k\neq n$.  So critical points of $V_{g_+}|_{\cC_1}$
are precisely solutions of $v^{(n)}(g)=c$.  The identification of the 
first variation of $v^{(2k)}(g)$ from \cite{ISTY, G2} holds just as well
for $2k=n$, so this 
gives immediately the second  variation of $V_{g_+}|_{\cC_1}$ in terms of
the tensor 
$L^{ij}_{(n/2)}(g)$ (Theorem~\ref{V2}).  Einstein metrics are critical
points for $V_{g_+}|_{\cC_1}$, and upon evaluating the second variation at   
an Einstein metric, we deduce the following analogue of
Theorem~\ref{signs}.  Different boundary components contribute
independently to 
the change in $V_{g_+}$, so we take $M$ to be connected and formulate the   
result for a general AHE manifold 
$(X,g_+)$ such that $\cC=(M,[g])$ is one of the connected components
of its conformal infinity.  We fix arbitrarily a representative of the
conformal infinity on each of the other connected components and view
$V_{g_+}$ as a function of the metric in the conformal class on $M$.

\begin{theorem}\label{VE}
Let $n\geq 2$ be even.  Let $(M,g)$ be a unit volume connected compact
Riemannian manifold  
with constant nonzero scalar curvature which is not isometric to $S^n$ with
the standard metric (normalized to have unit volume).  If
$n\geq 4$, assume that $g$ is Einstein.  Let $(X,g_+)$ be AHE and suppose
that $(M,[g])$ is one of the connected components of its conformal
infinity.  The second 
variation of $V_{g_+}|_{\cC_1}$ is a definite quadratic form on  
$T_g\cC_1$ whose sign is as follows:
\begin{itemize}
\item
If $R<0$, then $\left(V_{g_+}|_{\cC_1}\right)''$ is negative definite.
\item
If $R>0$, then $\left(V_{g_+}|_{\cC_1}\right)''$ is positive definite if 
$n\equiv 0 \mod 4$ and negative definite if $n\equiv 2 \mod 4$.  
\end{itemize}
For $S^n$ with the (normalized) standard metric, 
$\left(V_{g_+}|_{\cC_1}\right)''$ is semi-definite with the indicated sign  
and with $n+1$-dimensional nullspace.  
\end{theorem}
We also state a sufficient condition for 
definiteness of the second variation of $V_{g_+}|_{\cC_1}$ for non-Einstein 
critical points which is analogous to the condition mentioned above for the 
$\cF_k$.   

Colin Guillarmou has informed us that he has proved Theorem~\ref{VE}
in joint work with S. Moroianu and J.-M. Schlenker.  

The results of~\cite{CF,G2} and Theorem~\ref{signs}
indicate the importance of the renormalized volume functionals in   
conformal geometry, which we will hopefully continue to explore in future
works.

\section{Renormalized Volume}\label{renormvol}

Let $g_+$ be an asymptotically hyperbolic Einstein (AHE) metric on $X^{n+1}$ 
with smooth conformal infinity $(M,[g])$, where $M=\pa X$.   
Let $g$ be a metric in the conformal class on $M$.  One can uniquely
identify  a neighborhood of $\pa X$ with $[0,\ep)\times\pa X$ so that 
$g_+$ takes the normal form 
\begin{equation}\label{normalform}
g_+=r^{-2}\left(dr^2+g_r\right) 
\end{equation}
for a 1-parameter family $g_r$ of metrics on $M$ with $g_0=g$. 
The defining function $r$ is called the geodesic defining function
associated to $g$.
A boundary regularity result (\cite{CDLS, H, BH}) shows that $g_r$ is
smooth up to $r=0$ if $n$ is odd, and has a 
polyhomogeneous expansion as $r\rightarrow 0$ if $n$ is even.  The family
$g_r$ is 
even to order $n$; in particular $\pa_r g_r|_{r=0}=0$.  Thus the geodesic 
compactification $\gb_{geod}=r^2g_+$ is totally geodesic.  Any other  
compactification which induces the same boundary metric can be written 
as $\gb=e^{2\om}\gb_{geod}$ for some $\om \in C^{\infty}({\Xb})$ 
satisfying $\om=0$ on $\pa X$.  Such a compactification $\gb$ is totally   
geodesic if and only if $\om = O(r^2)$.  

The renormalized volume coefficients $v^{(2k)}(g)$ are defined for 
$0\leq k\leq \lfloor n/2\rfloor$ by the asymptotic expansion  
\begin{equation}\label{notation1}
\left(\frac{\det g_r}{\det g}\right)^{1/2}
= \sum_{k=0}^{\lfloor n/2\rfloor} v^{(2k)}(g)r^{2k} + o(r^{n}).  
\end{equation}
If $n$ is odd, the definition can be extended to all $k\geq 0$ by
considering metrics $g_+$ of the form \eqref{normalform} for which $g_r$ is
even to infinite order and for which $\Ric(g_+)+ng_+$ vanishes to infinite
order.  The $v^{(2k)}(g)$ are local curvature invariants of $g$ 
which are determined by an inductive algorithm; see \cite{G1, G2}.  Clearly
$v^{(0)}(g)=1$.  The next three are given by: 
\[
\begin{split}
v^{(2)}(g)  &= -\frac{R}{4(n-1)}\\
v^{(4)}(g) 
&= \frac 14 \sigma_{2}(g^{-1}P)=\frac18\left[(P^j{}_j)^2-P^{ij}P_{ij}\right]\\
v^{(6)}(g) & =  -  \frac18 \left[\sigma_{3}(g^{-1}P) + 
   \frac{1}{3(n-4)}P^{ij}B_{ij}\right],
\end{split}
\]  
where $P_{ij}:=\frac{1}{n-2}\left[R_{ij}-Rg_{ij}/2(n-1)\right]$ and   
$B_{ij} : =\nabla^k\nabla_kP_{ij}-\nabla^k\nabla_jP_{ik}-P^{kl}W_{kijl}$    
are the Schouten and Bach tensors of $g$, and $\sigma_k(g^{-1}P)$ is the
$k$-th elementary symmetric function of the eigenvalues of the endomorphism 
$g^{-1}P$.  


The rest of this section is devoted to the proof of
Theorem~\ref{geodcomp}.  The first step is to establish the 
result for a specific totally geodesic 
compactification.  Let $g$ be a metric in the conformal class at infinity
with associated geodesic defining function $r$ and geodesic
compactification $r^2g_+=dr^2+g_r$.  
Theorem 4.1 of \cite{FG1} asserts the existence of a unique 
$U\in C^{\infty}(\Xint)$ such that $-\Delta_{g_+}U=n$ (our convention is
$\Delta=\nabla^k\nabla_k)$ with the asymptotics 
\begin{equation}\label{U}
U=\log r +A +Br^n,
\end{equation}
where $A$, $B\in C^{\infty}(\Xb)$ are even functions  modulo
$O(r^{\infty})$ and $A|_{\partial X}=0$.  Then $e^U=re^{A+Br^n}$ is a
defining function and  
\begin{equation}\label{gUdef}
\gb_U:= e^{2U}g_+=e^{2(A+Br^n)}\left(dr^2+g_r\right)
\end{equation}
is a totally geodesic compactification.  Theorem 4.3 of \cite{FG1} asserts
that 
\begin{equation}\label{V}
V_{g_+}=\int_{\pa X}B|_{\pa X}\,dv_g.  
\end{equation}
\begin{proposition}\label{Ucomp}
Theorem~\ref{geodcomp} holds for $\gb=\gb_U$.
\end{proposition}

\noindent
Proposition~\ref{Ucomp} follows from an argument of \cite{CQY}.  The
formula of \cite{CQY} mentioned in the introduction for the renormalized
volume in terms of the 
the Euler characteristic and the integral of a pointwise
conformal invariant is derived by applying  
Alexakis' theorem \cite{Al} on the existence 
of a decomposition of $Q$-curvature.  Our proof of 
Proposition~\ref{Ucomp} uses an analogous identity expressing the
$Q$-curvature as a multiple of $v^{(n+1)}$ and a divergence.  The       
existence of such a formula can be deduced by general considerations since
the integrals of the $Q$-curvature and $v^{(n+1)}$ agree up to a
multiplicative constant on compact Riemannian manifolds.  However, an
explicit formula of this 
kind is known:  the holographic formula for $Q$-curvature of \cite{GJ}.  So
we base our proof on the holographic formula for $Q$-curvature.      

We first recall some properties of the 
metric $\gb_U$ which were established in \cite{CQY}. 

\begin{proposition}\label{0630}
Let $\gb_U$ be given by \eqref{gUdef}, where $U$ is the solution of
$-\Delta_{g_+}U=n$ with asymptotics \eqref{U} as above. Then we have   
\begin{itemize}
\item (\cite{CQY} Lemma 2.1)
\begin{equation}\label{P1}   
Q({\gb_U})=0.
\end{equation}

\item (\cite{CQY} Lemma 3.1)
Let $R$ denote the scalar curvature and $\Delta$ the Laplacian for the
metric $\gb_U$.  Then   
\begin{equation}\label{P2} 
\pa_r\Delta^{(n-3)/2}R=-2n n!B \,\,\, \text{on} \,\,  \pa X.
\end{equation}

\item (\cite{CQY} Lemma 3.2) Let $*$ stand for indices in the tangential
  directions on $\pa X$. For the covariant derivatives of the curvature
  tensor $R_{ijkl}$ of $\gb_U$, the following three types of components  
$$
\nabla_r^{2k+1} R_{****},
\quad \nabla_r^{2k} R_{r***}, \quad
\nabla_r^{2k-1} R_{r* r*},
$$ 
vanish at the boundary for $1 \leq 2k+1 \leq n-2$.
\end{itemize}
\end{proposition}

\vskip .2in
\noindent
{\it Proof of Proposition~\ref{Ucomp}}.  
The holographic formula for $Q$-curvature states that for any metric in 
even dimension $m=n+1$, one has 
\begin{equation}\label{Qform}
2c_{m/2}Q=v^{(n+1)} 
+\frac{1}{n+1}\sum_{k=1}^{(n-1)/2}(n+1-2k)p_{2k}^*v^{(n+1-2k)},
\end{equation}
where $c_l^{-1}=(-1)^l2^{2l}l!(l-1)!$.  Here the $v^{(n+1-2k)}$ are the
renormalized volume coefficients, $p_{2k}$ is a natural
differential operator of order $2k$ with no constant term and with
principal part $a_{n+1,k}\Delta^k$, where
$$
a_{n+1,k}=\frac{\Gamma\left((n+1-2k)/2\right)}
{2^{2k}\,k!\;\Gamma\left((n+1)/2\right)},
$$
and $p_{2k}^*$ denotes the formal adjoint of $p_{2k}$.  In particular, each
term $p_{2k}^*v^{(n+1-2k)}$ with $k\geq 1$ in the sum on the right-hand
side of \eqref{Qform} is the divergence of a natural 1-form.  

Apply \eqref{Qform} to $\gb_U$ and use \eqref{P1} to deduce that  
\[
\begin{split}
v^{(n+1)}(\gb_U) = &
-\frac{1}{n+1}\sum_{k=1}^{(n-1)/2}(n+1-2k)p_{2k}^*v^{(n+1-2k)}\\
= & -\frac{2}{n+1}p_{n-1}^*v^{(2)}
-\frac{1}{n+1}\sum_{k=1}^{(n-3)/2}(n+1-2k)p_{2k}^*v^{(n+1-2k)}\\
= & -\frac{2}{n+1}a_{n+1,(n-1)/2}\Delta^{(n-1)/2}v^{(2)}\\ 
& \qquad +q\, v^{(2)}
-\frac{1}{n+1}\sum_{k=1}^{(n-3)/2}(n+1-2k)p_{2k}^*v^{(n+1-2k)}, 
\end{split}
\]
where $q$ is a natural differential operator of order less than 
$n-1$ which is a divergence, and all the terms on the right-hand side refer
to the metric $\gb_U$.  Now integrate over $X$.  The right-hand side is a
divergence so can be rewritten as a boundary integral.  Recalling that 
$v^{(2)}=-\frac12 P^k{}_k = -\frac{1}{4n} R$, one has
$$
\int_X\Delta^{(n-1)/2}v^{(2)}\,dv_{\gb_U}=
\frac{1}{4n}\int_{\pa X}\pa_r\Delta^{(n-3)/2}R\; dv_g.
$$
But \eqref{P2} asserts that 
$\pa_r\Delta^{(n-3)/2}R=-2n n!B$ on $\pa X$.  Substituting and using 
\eqref{V} gives
$$
\int_X\Delta^{(n-1)/2}v^{(2)}\,dv_{\gb_U}=-\frac{n!}{2}\, V_{g_+}.  
$$
All terms in the expression
$$
q\, v^{(2)}
-\frac{1}{n+1}\sum_{k=1}^{(n-3)/2}(n+1-2k)p_{2k}^*v^{(n+1-2k)}
$$
involve fewer derivatives of $\gb_U$.  Arguing as in \cite{CQY}, the third
part of Proposition~\ref{0630} implies that   
the resulting integral over the boundary vanishes.  Thus
$$
\int_X v^{(n+1)}(\gb_U)\,dv_{\gb_U} 
= -\frac{2}{n+1}a_{n+1,(n-1)/2}\left(-\frac{n!}{2}\right) V_{g_+}.  
$$
Collecting the constant gives the result.  
\stopthm

\medskip
\noindent
{\it Proof of Theorem~\ref{geodcomp}.}
Let $\gb$ be a totally geodesic compactification of $g_+$ with induced
boundary metric $g$.  
Let $\gb_U$ be the compactification as above with the same boundary
metric.  Then $\gb=e^{2\om} \gb_U$ where $\om = O(r^2)$.  We will show that  
\begin{equation}\label{equal}
\int_Xv^{(n+1)}(\gb)\,dv_{\gb}=\int_Xv^{(n+1)}(\gb_U)\,dv_{\gb_U}.  
\end{equation}
The result then follows by Proposition~\ref{Ucomp}.  

Set $\gb_t=e^{2t\om}\gb_U$.  
Theorem 1.5 of \cite{G2} gives a divergence formula of the form   
$$
\pa_t \left(v^{(n+1)}(\gb_t)\,dv_{\gb_t}\right)
=(-2)^{-(n+1)/2}\nabla_i
\left(L^{ij}_{((n+1)/2)}(\gb_t)\nabla_j\,\om\right)\,dv_{\gb_t}  
$$
for a particular natural symmetric 2-tensor $L^{ij}_{((n+1)/2)}$.    
The covariant derivatives refer to the Levi-Civita connection of 
$\gb_t$.  Integrating by parts and using $\nabla \om|_{\pa X} =0$ gives
$$
\pa_t \int_Xv^{(n+1)}(\gb_t)\,dv_{\gb_t}=0.
$$
Thus $\int_Xv^{(n+1)}(\gb_t)\,dv_{\gb_t}$ is independent of $t$, which 
gives \eqref{equal}.

We have thus completed the proof of Theorem~\ref{geodcomp}.
\stopthm

\section{Second Variation}\label{secondvar}

If $M$ is a connected compact manifold, consider the functional    
\begin{equation}\label{def1}
\cF_k(g)=(-2)^{k}\int_Mv^{(2k)}(g)\,dv_g
\end{equation}
on the space of metrics on $M$, where $v^{(2k)}(g)$ is the renormalized  
volume coefficient defined in \S\ref{renormvol}. For notational
convenience, we set 
\begin{equation}\label{notation2}
v_k(g)= (-2)^k v^{(2k)}(g).  
\end{equation}
This is the same notation as in \cite{G2}. The coefficient is chosen 
so that  if $g$ is locally conformally flat, then 
$$
v_k(g)=\sigma_k\left(g^{-1}P\right),\qquad 0\leq k\leq n
$$
(see Proposition 1 of \cite{GJ}).  
It will also be convenient to introduce $\rho =-\frac12 r^2$ and 
$g(\rho)=g_r$, where $g_r$ is the 1-parameter family of metrics appearing
in \eqref{normalform}.  Then the expansion \eqref{notation1} defining the
$v^{(2k)}$ becomes 
\begin{equation}\label{vkdef}
\left(\frac{\det g(\rho)}{\det g}\right)^{1/2}
\sim  \sum_{k\geq 0} v_k(g)\rho^k. 
\end{equation}
Set $v(\rho)=\left(\det g(\rho)/\det g\right)^{1/2}$.  

Recall that for even $n$, the $v^{(2k)}(g)$ are only defined for $k\leq
n/2$ for general metrics.  But they are invariantly defined
for all $k$ if $n\geq 4$ and 
$g$ is locally conformally flat or conformally Einstein.  This is because
in these cases there is an invariant determination of $g(\rho)$
to all orders; see \cite{FG2}.  If $g$ is Einstein with 
$R_{ij}=2a (n-1)g_{ij}$, one has $g(\rho)=(1+a \rho)^2g$.  Observe that this
gives $v(\rho)=(1+a \rho)^n$, so 
\begin{equation}\label{vkEinstein}
v_k(g)=a^k\begin{pmatrix}n\\k\end{pmatrix},\qquad 0\leq k\leq n.
\end{equation}
For a conformal rescaling $\gh=e^{2\om}g$ of an Einstein metric $g$,
$\gh_r$ is defined by putting the Poincar\'e metric
$r^{-2}\left(dr^2+(1-a r^2/2)^2g\right)$ for $g$ into normal form
relative to $\gh$ by a diffeomorphism.  Then one sets $\gh(\rho)=\gh_r$
with $\rho = -\frac12 r^2$ 
and defines $v_k(\gh)$ via \eqref{vkdef}.  This is
well-defined, but a direct formula in terms of $\gh$ is not available.  

The crucial ingredient in the variational analysis is the following formula
for the conformal variation of the $v_k(g)$.  
For a Riemannian manifold $(M,g)$  and $\om\in C^\infty(M)$, set 
$g_t=e^{2t\om}g$ and 
define $\delta v_k(g,\om)=\pa_t|_{t=0}v_k(g_t)$.  Then (2.4), (3.8) of
\cite{ISTY} (see also Theorem 1.5 of \cite{G2}) show that    
\begin{equation}\label{varvk}
\delta
v_k(g,\om)=\nabla_i\left(L^{ij}_{(k)}(g)\nabla_j\om\right)-2kv_k(g)\om, 
\end{equation}
where
\begin{equation}\label{Ldef}
L^{ij}_{(k)}(g)= 
-\frac{1}{k!}\pa_\rho^k\left( v(\rho)\int_0^\rho g^{ij}(u)\,du
\right)\Big{|}_{\rho =0}
=-\sum_{l=1}^k\frac{1}{l!}v_{k-l}(g)\,
\pa_\rho^{l-1}g^{ij}(\rho)|_{\rho =0}.     
\end{equation}
Here $g^{ij}(u)=\left(g_{ij}(u)\right)^{-1}$ and $\nabla$ denotes 
the covariant derivative with respect to $g$.     

We first review the identification of the critical points of $\cF_k$ from
\cite{CF,G2}.   
By \eqref{notation2}, \eqref{def1} can be re-written as 
$$
\cF_k(g)=\int_Mv_{k}(g)\,dv_g.
$$  
By \eqref{varvk} and the fact that $\delta dv_g = n\om dv_g$, one deduces
that the conformal variation of $\cF_k$ is given by 
\begin{equation}\label{varFk}
\delta\cF_k=(n-2k)\int_Mv_k\om\,dv_g.
\end{equation}
For $n$ even and $2k=n$, this recovers the fact that $\cF_{n/2}$ is
conformally invariant.  For $2k\neq n$,  
we are interested in the restriction $\cF_k|_{\cC_1}$, where $\cC_1$
denotes the space of unit volume metrics in a conformal class $\cC$ of
metrics on $M$.  We use the  Lagrange   
multiplier method.  The critical points are the
metrics $g\in \cC_1$ which satisfy for some constant $\lambda$ that 
$$
\delta\left(\cF_k -\lambda\Vol(M)\right)(g,\om)=0\quad\text{for
  all }\om.  
$$
By \eqref{varFk}, this is
$$
(n-2k)\int_Mv_k\om\,dv_g-n\lambda \int_M\om\,dv_g=0 \quad\text{for
  all }\om,
$$
which gives $v_k(g)=n\lambda/(n-2k)$.  Thus the critical points are 
precisely the unit volume metrics for which $v_k(g)$ is constant.    

The following theorem identifies the second variation of 
$\cF_k|_{\cC_1}$ at a critical point.  Suppose $g$ is a unit volume metric
for which $v_k(g)$ is constant.  The
tangent space of $\cC_1$ at $g$ is given by
$$
T_g\cC_1=\left\{2\om g:\int_M \om\,dv_g=0\right\}.  
$$
For such an $\om$, set 
$$
\left(\cF_k|_{\cC_1}\right)''(\om)
=\pa_t^2|_{t=0}\cF_k(\gamma_t), 
$$
where $\gamma_t$ is a curve in $\cC_1$ satisfying $\gamma_0=g$ and
$\gamma_0'=2\om g$.    
\begin{theorem}\label{hess}
Let $n\geq 3$, $k\geq 1$ and $k\leq n/2$ if $n$ is even.  Let $(M,g)$ be a
connected compact  
Riemannian manifold and suppose $g$ satisfies $v_k(g)=c$ for some constant
$c$.  Let $\om\in C^\infty(M)$ satisfy $\int_M \om\,dv_g=0$. Then 
$$
\left(\cF_k|_{\cC_1}\right)''(\om)=
-(n-2k)\int_M\left[L^{ij}_{(k)}(g)\om_i\om_j+2kv_k(g)\om^2\right]\,dv_g. 
$$
\end{theorem}
\begin{proof}
We can assume that $k\neq n/2$.  
Define $\lambda$ by $c=n\lambda/(n-2k)$, so that 
$\delta\left(\cF_k -\lambda\Vol(M)\right)=0$.  
Since the Hessian at a critical point is invariantly
defined on the tangent space, we have
$$
\left(\cF_k|_{\cC_1}\right)''(\om)=
\pa_t^2|_{t=0}\left(\cF_k-\lambda\Vol(M)\right)(g_t),
$$
with $g_t=e^{2t\om}g$ as above.  Now \eqref{varFk} gives 
$$
\pa_t\cF_k(g_t)=(n-2k)\int_Mv_k(g_t)\om\,dv_{g_t}.
$$
Combining this with
$$
\pa_t\Vol_{g_t}(M)=n\int_M\om\,dv_{g_t}
$$
shows that
\[
\begin{split}
\pa_t^2|_{t=0}&\left(\cF_k-\lambda\Vol(M)\right)(g_t) 
=\pa_t|_{t=0}\left[(n-2k)\int_Mv_k(g_t)\om\,dv_{g_t}
-\lambda n\int_M\om\,dv_{g_t}\right] \\
=&(n-2k)\int_M\left[\delta v_k(g,\om)+nv_k(g)\om\right]\om\,dv_g
-\lambda n^2\int_M\om^2\,dv_g\\
=&(n-2k)\int_M\delta
v_k(g,\om)\om\,dv_g+n^2\lambda\int_M\om^2\,dv_g
-\lambda n^2\int_M\om^2\,dv_g\\  
=&(n-2k)\int_M\delta v_k(g,\om)\om\,dv_g\\
=&-(n-2k)\int_M\left[L^{ij}_{(k)}(g)\om_i\om_j+2kv_k(g)\om^2\right]\,dv_g, 
\end{split}
\]
where for the last equality we use \eqref{varvk} and integration by parts. 
\end{proof}

We remark that Theorem~\ref{hess} and its proof remain valid for all $k\geq
1$ when $n\geq 4$ is even if $g$ is Einstein or locally conformally flat.
This is because the main ingredient, \eqref{varvk}, just uses that the  
Poincar\'e metrics arising from conformally related metrics on the
boundary are related by a diffeomorphism.    

\bigskip
\noindent
{\it Proof of Theorem~\ref{signs}.}
Let $g$ be Einstein with $R_{ij}=2a (n-1)g_{ij}$.  We use
Theorem~\ref{hess} to evaluate $\left(\cF_k|_{\cC_1}\right)''$.  Recall
that $g_{ij}(\rho)=(1+a \rho)^2g_{ij}$ and $v(\rho)=(1+a \rho)^n$.  So 
$g^{ij}(\rho)=(1+a\rho)^{-2}g^{ij}$.  Hence
$$
v(\rho)\int_0^\rho g^{ij}(u)\,du=\rho(1+a\rho)^{n-1}g^{ij}.
$$
Therefore \eqref{Ldef} gives 
\begin{equation}\label{Leinstein}
L^{ij}_{(k)}(g)=-a^{k-1}\begin{pmatrix}n-1\\k-1\end{pmatrix} g^{ij},\qquad
  1\leq k\leq n.
\end{equation}
Recalling \eqref{vkEinstein}, Theorem~\ref{hess} gives 
\[
\begin{split}
\left(\cF_k|_{\cC_1}\right)''(\om)
=&(n-2k)a^{k-1}\begin{pmatrix}n-1\\k-1\end{pmatrix}
\int_M\left(|\nabla\om|_g^2-2na\om^2\right)\,dv_g\\
=&(n-2k)a^{k-1}\begin{pmatrix}n-1\\k-1\end{pmatrix}
\int_M\left(|\nabla\om|_g^2-R\om^2/(n-1)\right)\,dv_g.
\end{split}
\]

If $R<0$, this has the same sign as the leading
coefficient $(n-2k)a^{k-1}$, which gives the desired conclusion.

If $R>0$, we use Obata's estimate~\cite{O} for the first eigenvalue of
$-\Delta$ for 
an Einstein metric:  $\lambda_1(-\Delta)\geq R/(n-1)$ with equality only
for $S^n$.  This leads to the desired result.  For $S^n$, the equality
holds if and only if $\om$ is  
an eigenfunction corresponding to $\lambda_1$.  This is the
$(n+1)$-dimensional space of infinitesimal conformal factors corresponding to
conformal diffeomorphisms.
\stopthm

It is possible to formulate a result also for non-Einstein critical 
points.  It is clear from Theorem~\ref{hess} that if $L^{ij}_{(k)}(g)$ is 
definite and $v_k(g)$ is a constant of the same sign, then  
$\left(\cF_k|_{\cC_1}\right)''$ is definite.  This generalizes the result
of Viaclovsky that negative $k$-admissible critical points are local
extrema when $k=1$ or $2$ or $g$ is locally conformally flat.  

Consider finally the second variation of the renormalized volume
when $n$ is even.  Let $(X,g_+)$ be AHE and let $\cC=(M,[g])$ be one of  
the connected components of its conformal infinity.  Fix a
representative of the conformal 
infinity on each of the other connected components and view $V_{g_+}(g)$ as
a function on $\cC$.
As discussed in the introduction, its conformal variation is  
$$
\delta V_{g_+} = \int_Mv^{(n)}(g)\om\,dv_g.
$$
Upon introducing a Lagrange multiplier exactly as above for $\cF_k$, one  
deduces that the critical points of $V_{g_+}|_{\cC_1}$ are the unit volume
metrics for which $v^{(n)}(g)$ is constant. For such a $g$ and for $\om$ 
satisfying $\int_M\om\,dv_g=0$, we define the second variation by
$$
\left(V_{g_+}|_{\cC_1}\right)''(\om)
=\pa_t^2|_{t=0}V_{g_+}(\gamma_t), 
$$
where $\gamma_t$ is a curve in $\cC_1$ satisfying $\gamma_0=g$ and
$\gamma_0'=2\om g$.    
\begin{theorem}\label{V2}
Let $n\geq 2$ be even.  Let $(X,g_+)$ be AHE and let 
$(M,[g])$ be one of the connected components of its conformal  
infinity.  Suppose that $g$ satisfies that $v^{(n)}(g)=c$ for some constant
$c$ and let $\int_M \om\,dv_g=0$. Then 
$$
\left(V_{g_+}|_{\cC_1}\right)''(\om)
=(-1)^{n/2+1}2^{-n/2}\int_M
\left[L^{ij}_{(n/2)}(g)\om_i\om_j+nv_{n/2}(g)\om^2\right]\,dv_g. 
$$
\end{theorem}
\begin{proof}
We argue exactly as in the proof of Theorem~\ref{hess}.  Define $\lambda$
by $c=n\lambda$ so that 
$\delta\left(V_{g_+} -\lambda\Vol(M)\right)=0$.  Then
$$
\left(V_{g_+}|_{\cC_1}\right)''(\om)=
\pa_t^2|_{t=0}\left(V_{g_+}-\lambda\Vol(M)\right)(g_t)
$$
with $g_t=e^{2t\om}g$.  And
\[
\begin{split}
\pa_t^2|_{t=0}&\left(V_{g_+}-\lambda\Vol(M)\right)(g_t)
=\pa_t|_{t=0}\left[\int_Mv^{(n)}(g_t)\om\,dv_{g_t}
-\lambda n\int_M\om\,dv_{g_t}\right] \\
=&\int_M\left[\delta v^{(n)}(g,\om)+nv^{(n)}(g)\om\right]\om\,dv_g
-\lambda n^2\int_M\om^2\,dv_g\\
=&(-2)^{-n/2}\int_M\delta v_{n/2}(g,\om)\om\,dv_g\\  
=&(-1)^{n/2+1}2^{-n/2}\int_M\left[L^{ij}_{(n/2)}(g)\om_i\om_j+nv_{n/2}(g)\om^2\right]\,dv_g. 
\end{split}
\]
\end{proof} 

\noindent
{\it Proof of Theorem~\ref{VE}.}  This follows exactly as in the proof of
Theorem~\ref{signs} above.  If $n\geq 4$ and $g$ is Einstein with
$R_{ij}=2a(n-1)g_{ij}$, or if $n=2$ and $g$ has constant scalar curvature
$R=4a$, then  $L^{ij}_{(n/2)}(g)$ is given by \eqref{Leinstein} and
$v_{n/2}(g)$ by \eqref{vkEinstein}.  Substituting into Theorem~\ref{V2}
gives
$$
\left(V_{g_+}|_{\cC_1}\right)''(\om)=
-(-a)^{n/2-1}2^{-n/2}\begin{pmatrix}n-1\\n/2-1\end{pmatrix}
\int_M\left(|\nabla\om|_g^2-R\om^2/(n-1)\right)\,dv_g.
$$
The conclusion is now clear if $R<0$.  If $R>0$, it follows from the same
argument as in the proof of Theorem~\ref{signs} using Obata's estimate on  
$\lambda_1(-\Delta)$.  
\stopthm

The sign of the second variation can also be deduced from Theorem~\ref{V2}
for certain non-Einstein critical points of $V_{g_+}|_{\cC_1}$.  It is
clear that $\left(V_{g_+}|_{\cC_1}\right)''$ is   
definite if $L^{ij}_{(n/2)}(g)$ is definite and $v_{n/2}(g)$ is a constant
of the same sign.  For instance, one concludes that 
$\left(V_{g_+}|_{\cC_1}\right)''$ is negative definite if $g$ is a negative 
$n/2$-admissible solution of $\sigma_{n/2}(g^{-1}P)=c$ and $n=4$ or
$n\geq 6$ with $g$ locally conformally flat.  Under these conditions,  
$v_{n/2}(g)=\sigma_{n/2}(g^{-1}P)$ and 
$L^{ij}_{(n/2)}(g)=-T^{ij}_{(n/2-1)}(g^{-1}P)$ is the negative of the
corresponding Newton tensor (see \cite{G2}).

\end{document}